\documentclass[reqno,10pt]{elsarticle}

\usepackage{amsmath, amsthm, amssymb, enumerate}
\usepackage{comment}
\usepackage{color}
\usepackage{datetime}
\usepackage{fancyhdr}
\usepackage{soul}
\usepackage{amsthm, amssymb}
\usepackage{amsmath}
\usepackage[margin=1in]{geometry}
 \def\pdot{{\color{purple}{\hskip-.0truecm\rule[-1mm]{4mm}{4mm}\hskip.2truecm}}\hskip-.3truecm}
 \def\cor{\color{red}}
 \def\cob{\color{black}}

\begin{document}
\def\dbar{\overline\partial}
\def\comm{L}
\def\intint{\int\!\!\!\!\int}
\def\intinttext{\int\!\!\!\int}
\def\intintint{\int\!\!\!\!\int\!\!\!\!\int}
\def\intintintint{\int\!\!\!\!\int\!\!\!\!\int\!\!\!\!\int}
\def\ques{{\cor \underline{??????}\cob}}
\def\nto#1{{\coC \footnote{\em \coC #1}}}
\def\fractext#1#2{{#1}/{#2}}
\def\fracsm#1#2{{\textstyle{\frac{#1}{#2}}}}   
\def\baru{U}
\def\nnonumber{}
\def\palpha{p_{\alpha}}
\def\valpha{v_{\alpha}}
\def\Ualpha{U_{\alpha}}
\def\qalpha{q_{\alpha}}
\def\walpha{w_{\alpha}}
\def\falpha{f_{\alpha}}
\def\dalpha{d_{\alpha}}
\def\galpha{g_{\alpha}}
\def\halpha{h_{\alpha}}
\def\plusdelta{+\delta}
\def\psialpha{\psi_{\alpha}}
\def\psibeta{\psi_{\beta}}
\def\betaalpha{\beta_{\alpha}}
\def\gammaalpha{\gamma_{\alpha}}
\def\Talpha{T}
\def\TTalpha{T_{\alpha}}
\def\TTalphak{T_{\alpha,k}}
\def\falphak{f^{k}_{\alpha}}

\def\vbar{\overline v}
\def\ubar{\overline u}
\def\pbar{\overline p}
\def\qbar{\overline q}
\def\wbar{\overline w}
\def\zbar{\overline z}

\newcommand{\bv}{{u}}
\newcommand{\R}{\mathbb{R}}
\newcommand{\N}{\mathbb{N}}
\newcommand{\Z}{\mathbb{Z}}
\newcommand{\Br}{B_r(x_0)}
\newcommand{\Qr}{Q_r(x_0,t_0)}
\newcommand{\hilight}[1]{\colorbox{yellow}{#1}}
%
\def\pdot{}
\def\Pdot{}
\def\tdot{\fbox{\fbox{\bf\tiny I'm here; \today \ \currenttime}}}
 \baselineskip=15pt
\def\nts#1{{\color{red}\hbox{\bf ~#1~}}} 
\def\mbar{{\overline M}}
\def\tilde{\widetilde}
\newtheorem{Theorem}{Theorem}[section]
\newtheorem{Corollary}[Theorem]{Corollary}
\newtheorem{Proposition}[Theorem]{Proposition}
\newtheorem{Lemma}[Theorem]{Lemma}
\newtheorem{Remark}[Theorem]{Remark}
\newtheorem{definition}{Definition}[section]
\def\theequation{\thesection.\arabic{equation}}
\def\endproof{\hfill$\Box$\\}
\def\square{\hfill$\Box$\\}
\def\comma{ {\rm ,\qquad{}} }            
\def\commaone{ {\rm ,\qquad{}} }         
\def\dist{\mathop{\rm dist}\nolimits}    
\def\sgn{\mathop{\rm sgn\,}\nolimits}    
\def\Tr{\mathop{\rm Tr}\nolimits}    
\def\div{\mathop{\rm div}\nolimits}    
\def\curl{\mathop{\rm curl}\nolimits}    
\def\supp{\mathop{\rm supp}\nolimits}    
\def\divtwo{\mathop{{\rm div}_2\,}\nolimits}    
\def\re{\mathop{\rm {\mathbb R}e}\nolimits}    
\def\indeq{\qquad{}\!\!\!\!}                     
\def\period{.}                           
\def\semicolon{\,;}                      
\def\TT{S}

\def\red#1{\textcolor{red}{#1}}
\def\mgt#1{\textcolor{magenta}{#1}}
\def\Lpnorm#1{\Vert #1\Vert_{L^p}}


\title{On the local existence for an active scalar equation \\ in critical regularity setting}

\author{Walter Rusin}
\address{Department of Mathematics, Oklahoma State University, Stillwater, OK 74078}

\author{Fei Wang}
\address{Department of Mathematics, University of Southern California, Los Angeles, CA 90089}

\begin{abstract}
	In this note, we address the local well-posedness for the active scalar equation $\partial_t \theta + u\cdot \nabla \theta =0$, where $u = - \nabla^\perp(-\Delta)^{-1+\beta/2}\theta$. The local existence of solutions in the Sobolev class $H^{1+\beta+\epsilon}$, where $\epsilon>0$ and $\beta \in (1,2)$, has been recently addressed in \cite{HKZ}. The critical case $\epsilon =0$ has remained open. Using a different technique, we prove the local well-posedness in the Besov space $B^{1+\beta}_{2,1}$, where $\beta \in (1,2)$. The proof is based on log-Lipschitz estimates for the transport equation. 
\end{abstract}

\maketitle

\section{Introduction}\label{sec:intro}
\setcounter{equation}{0}

We address the problem of local existence of solutions for the family of active scalar equations 
\begin{align}\label{eq1}
	\partial_t \theta + u \cdot \nabla \theta =0
\end{align}
where the drift velocity $u$ is obtained from the scalar $\theta$ through a singular integral operator $u=M(\theta)$. Equations of that type have been recently the focus of interest, motivated by the ubiquity of physical contexts in which such equations arise. In \cite{CCCGW} (see also \cite{CW}, \cite{CCW} and \cite{HKZ}) the authors introduce a family of equations of the type \eqref{eq1} with
\begin{align}\label{eq2}
	 u = \nabla^\perp \psi \nonumber \\
	 \Delta \psi = \Lambda^\beta \theta 
\end{align}
where $\Lambda = (-\Delta)^{1/2}$ is a Calder\'on-Zygmund operator defined as a Fourier multiplier with the symbol $|\xi|$ and $\nabla^\perp \psi = (-\partial_2 \psi, \partial_1 \psi)$. Equation \eqref{eq1} together with \eqref{eq2} reduces to the Euler equation \cite{BM} in vorticity formulation when $\beta =0$, whereas for $\beta =1$ we obtain the non-dissipative surface quasigeostrophic equation (SQG) \cite{CMT}. For $\beta =2$, the equation bears resemblance with the non-dissipative magneto-geostrophic equation \cite{FRV}. In this note we are concerned with the local existence of solutions for \eqref{eq1}--\eqref{eq2} with $1 < \beta < 2$. In particular, note that the drift velocity is more singular than in the case of SQG. 

In \cite{CCCGW}, the authors establish the local wellposedness of \eqref{eq1}--\eqref{eq2} in Sobolev spaces $H^m$ with an integer $m \geq 4$. The proof is based on the use of the anti-symmetry of the nonlinearity and a commutator estimate in Sobolev spaces with low (negative) regularity. In \cite{HKZ}, the authors note that from the formulation of \eqref{eq1} it seems that the optimal space for existence of solutions should be the space which yields a uniformly Lipschitz velocity field $u$. From Sobolev embedding we then obtain that a sufficient condition for $u$ to be in the desired class is $\theta \in H^{1+\beta+\epsilon}$, where $\epsilon >0$ is arbitrary. Moreover, since the considered equation is not equipped with any regularizing effects, the regularity of initial data should be preserved at least locally in time thus it appears that for local well-posedness we only require that $\theta_0 \in H^{1+\beta+\epsilon}$. The proof is based on a new commutator estimate, improving the estimates presented in~\cite{CCCGW}. The local existence of solutions in the critical space $H^{1+\beta}$ has remained an open problem. 

In this note we address the existence of local solutions in the Besov spaces with critical regularity, namely $B^{1+\beta}_{2,1}$. Our proof is based on the log-Lipschitz estimate for the transport equation rather than the use of usual commutator estimates.

\section{Preliminaries}

For the sake of clarity of notation let us briefly recall the diadic decomposition of the frequency space $\mathbb{R}^n$. 
Let $\phi, \psi  \in \mathcal{S}(\mathbb{R}^n)$  be two radial functions with supports in Fourier space
\begin{align*}
	 supp\; \widehat{\phi} \subset \{ \xi: |\xi| \leq 4/3\}, \;\;\;\;\; supp\; \widehat{\psi} \subset \{\xi: 3/4 < |\xi| < 8/3 \}, \\
\end{align*}
and with the property that
\begin{align*}
	 \widehat{\phi(\xi)} + \sum_{j \geq 0}\widehat{\psi} (2^{-j} \xi) = 1
\end{align*}
 for all $\xi \in \mathbb{R}^n$.
The existence of such functions is a classical result and we refer the reader to \cite{BCD} (or other similar sources) for more details. 

In order to isolate the interactions of different Fourier modes, we define the Littlewood-Paley projections $\Delta_j$ for $j \in \mathbb{Z}$ as follows
\begin{align*}
	\Delta_j f(x) = \begin{cases}
		0 & \text{ if $j \leq -2$} \\
		\int_{\mathbb{R}^n} \phi(y)f(x-y)\;dy & \text{ if $j = -1$} \\
		2^{jn}\int_{\mathbb{R}^n} \psi(2^jy)f(x-y)\;dy & \text{ if $j \geq 0$} \end{cases}
\end{align*}
For $j \in \mathbb{Z}$ the operator $S_j$ is the sum of $\Delta_k$ with $k \leq j-1$, that is 
\begin{equation*}
	S_j f(x) = \Delta_{-1}f(x) + \Delta_0 f(x) + \ldots + \Delta_{j-1}f(x) = \int_{\mathbb{R}^n} \phi(2^jy)f(x-y)\;dy. 
\end{equation*}
It is by now a classical result that for any tempered distribution $f$ we have $S_jf \to f$ in the distributional sense as $j \to \infty$. 

For any $s \in \mathbb{R}$ and $p,q \in [1,\infty]$, the Besov space $B^s_{p,q}$ consists of all tempered distributions $f \in \mathcal{S}'(\mathbb{R}^n)$ such that the sequence $\{ 2^{js}\|\Delta_j f\|_{L^p}\}$ is summable in the sense of $\ell^q(\mathbb{Z})$ with the obvious modification in the case $\ell^\infty(\Z)$. 



\section{The main result and the commutator estimate}

Let $\beta \in (1,2)$. we consider the problem \eqref{eq1}--\eqref{eq2} with initial data
\begin{align}
	\theta(x,0) = \theta_0(x).
\end{align}
The main result of this note is the following
\begin{Theorem}\label{thm1}
	Let $\theta_0 \in B^{1+\beta}_{2,1}(\R^2)$. There exists $T = T(\|\theta_0\|_{B^{1+\beta}_{2,1}})>0$ and a unique solution $\theta \in C([0,T], B^{1+\beta}_{2,1}(\R^2))$. 
\end{Theorem}

In order to prove our main result, we first recall a commutator estimate and an embedding lemma (cf.~\cite{Vi}).

Let $\alpha\in[0,1]$. As in~\cite{Vi}, we introduce the space $LL_\alpha$ of bounded functions on $\mathbb{R}^n$ such that 
  \begin{equation}
  \label{}
  \Vert f\Vert_{LL_{\alpha}}
  =
  \Vert f\Vert_{L^\infty}
  +
  \sup_{0<|x-y|\le1}\frac{f(x)-f(y)}{|x-y|(1-\log_2|x-y|)^\alpha}
  <\infty.
  \end{equation}
We set $LL=LL_0$.
\begin{Lemma}\label{lem1}
Let $j \geq -1$ be an integer and $p,q \in [1,\infty]$. Assume that $f\in L^p(\mathbb{R}^n)$ is a scalar function and  
$u\in LL_{1/q'}(\mathbb{R}^n)$ is a vector valued function, where $q'$ is the conjugate of $q$. We have
\begin{align*}
 \Vert S_{j-1}u \cdot \nabla \Delta_j f-\Delta_j(u \cdot \nabla f)\Vert_{L^p}
  \le
  C2^j \Vert u\Vert_{LL_{1/q'}}\sum_{j'\ge j-M}2^{-j'}(j'+1)^{1/q'}\Vert \Delta_{j'}f\Vert_{L^p}\,ds,
\end{align*}
for some universal constant $M>0$.
\end{Lemma}

We also need the following embedding lemma for the proof.
\begin{Lemma}\label{lem2}
	Let $p,q \in[1,\infty]$ and $sp= n$ where $n$ is the dimension. There exists a constant $C >0 $ such that 
	\begin{align}
		\|f\|_{LL_{1/q'}(\R^n)} \leq C \|f\|_{B^{1+s}_{p,q}(\R^n)}.
	\end{align}
	where $q'$ is the conjugate of $q$.
\end{Lemma}

\begin{proof}[Proof of Theorem \ref{thm1}]
Here, we restrict the proof to the \emph{a priori} estimate. The actual proof of existence is a result of a standard approximation procedure. 
	Let $j \geq -1$. We apply the Littlewood-Paley projection $\Delta_j$ to both sides of \eqref{eq1} and get
	\begin{align}\label{eq4.3}
		\partial_t \Delta_j \theta+ S_{j-1}u \cdot \nabla \Delta_j \theta 
		=  
		S_{j-1}u \cdot \nabla \Delta_j \theta-\Delta_j(u \cdot \nabla\theta).
	\end{align}
Multiplying the above equation by $\Delta_j \theta$ and integrating it, we obtain 
  \begin{align}
  \label{}
  \frac{1}{2}\frac{d}{dt}\Vert \Delta_j\theta\Vert_{L^2}^2
  +
  \int S_{j-1}u \cdot \nabla \Delta_j \theta\Delta_j \theta\,dx
  =
  \int (S_{j-1}u \cdot \nabla \Delta_j \theta-\Delta_j(u \cdot \nabla\theta))\Delta_j \theta\,dx.
  \end{align}
We apply H\"older's inequality to the term on the right side of the above equation in order to deduce
  \begin{align}
  \label{}
  \frac{d}{dt}\Vert \Delta_j\theta\Vert_{L^2}
  \le
  \Vert S_{j-1}u \cdot \nabla \Delta_j \theta-\Delta_j(u \cdot \nabla\theta)\Vert_{L^2},
  \end{align}
where we used the fact 
  \begin{equation}
  \label{}
  \int S_{j-1}u \cdot \nabla \Delta_j \theta\Delta_j \theta\,dx=0.
  \end{equation}
due to the divergence-free condition of $u$.
Integrating over time and using Lemma~\ref{lem1} gives
  \begin{align}
  \label{}
  \Vert \Delta_j\theta\Vert_{L^2}
  &\le
  \Vert \Delta_j\theta_0\Vert_{L^2}
  +
  \int_0^t  \Vert S_{j-1}u \cdot \nabla \Delta_j \theta-\Delta_j(u \cdot \nabla\theta)\Vert_{L^2} \,ds
  \nonumber\\&
  \le
  \Vert \Delta_j\theta_0\Vert_{L^2}
  +
  C\int_0^t 2^j \Vert u\Vert_{LL}\sum_{j'\ge j-M}2^{-j'}\Vert \Delta_{j'}\theta\Vert_{L^2}\,ds,
  \end{align}
from where by the discrete Young inequality we further get
  \begin{align}
  \label{}
  \Vert\theta \Vert_{B^{1+\beta}_{2,1}}
  &\le
  \Vert\theta_0 \Vert_{B^{1+\beta}_{2,1}}
  +
  C\int_0^t \Vert u\Vert_{LL}\sum_{j}\sum_{j'\ge j-M}2^{(2+\beta)(j-j')}2^{j'(1+\beta)}\Vert \Delta_{j'}\theta\Vert_{L^2}\,ds,
  \nonumber\\&
  \le
  \Vert\theta_0 \Vert_{B^{1+\beta}_{2,1}}
  +
  C\int_0^t \Vert u\Vert_{LL}\Vert\theta \Vert_{B^{1+\beta}_{2,1}}\,ds.
  \end{align}
	Applying Lemma \ref{lem2} with $q'=\infty$ and $p=2$, we obtain
	\begin{align*}
		\Vert u\Vert_{LL}
		\le
		C\Vert u\Vert_{B^{2}_{2,\infty}}
		\le
		C\Vert u\Vert_{B^{2}_{2,1}}.
	\end{align*}
By Bernstein's inequality, we have
  \begin{align}
  \label{}
  \Vert u\Vert_{B^{2}_{2,1}}
  &=
  \sum_{j=-1}^\infty 2^{2j}\Vert \Delta_ju\Vert_{L^2}
  =
  \sum_{j=-1}^\infty 2^{2j}\Vert \Delta_j\nabla^{\perp}(-\Delta)^{-1+\beta/2}\theta\Vert_{L^2}
  \nonumber\\&
  \le
  C\sum_{j=-1}^\infty 2^{j(1+\beta)}\Vert \Delta_j\theta\Vert_{L^2}=C\Vert\theta \Vert_{B^{1+\beta}_{2,1}}.
  \end{align}
Note that this yields
	\begin{align*}
		\|\theta(\cdot,t)\|_{B^{1+\beta}_{2,1}} \leq \|\theta_0\|_{B^{1+\beta}_{2,1}} 
		+ 
		C \int_0^t\| \theta \|_{B^{1+\beta}_{2,1}}^2 \;d\tau.
	\end{align*}
By Gr\"onwall's inequality we obtain
\begin{align}\label{estt1}
	\|\theta(\cdot,t)\|_{B^{1+\beta}_{2,1}} \leq C \| \theta_0\|_{B^{1+\beta}_{2,1}}\exp\left(C\int_0^t \|\theta(\cdot,\tau)\|_{B^{1+\beta}_{2,1}}\;d\tau\right)
\end{align}
Denoting by $V(t) = \int_0^t \|\theta(\cdot,\tau)\|_{B^{1+\beta}_{2,1}}\;d\tau$, we find
\begin{align}
	V(t) = C \int_0^t \|\theta(\cdot,\tau)\|_{B^{1+\beta}_{2,1}}\;d\tau \leq Ct \|\theta_0\|_{B^{1+\beta}_{2,1}}\exp(CV(t)).
\end{align}
Since $V(t)$ depends continuously on $t$ and $V(0)=0$, there exist $C_0,\eta >0$, such that 
\begin{align}\label{estt2}
	Ct\|\theta_0\|_{B^{1+\beta}_{2,1}} \leq \eta, \text{ then } CV(t) \leq C_0.
\end{align}
This, combined with 
\eqref{estt1} gives
\begin{align}
	\|\theta(t)\|_{B^{1+\beta}_{2,1}} \leq C\|\theta_0\|_{B^{1+\beta}_{2,1}}
\end{align}
for $t \in [0,T]$ for some $T>0$ (see \cite{HK}). We point out that \eqref{estt2} implies that the time of existence is bounded from below 
\begin{align}
	T > (C \|\theta_0\|_{B^{1+\beta}_{2,1}})^{-1}. 
\end{align}
This completes the proof. 
\end{proof}


\noindent \footnotesize{\bf Acknowledgments.} 
W.R. was supported in part by the NSF grant DMS-1311964 while F.W. was supported in part by the NSF grant DMS-1311943.
\normalfont
\normalsize

\end{document}